\documentclass[
final,
]{amsart}

\usepackage[english]{babel}
\usepackage[utf8]{inputenc}

\usepackage[T1]{fontenc}
\usepackage{lmodern}
\usepackage[section]{placeins}
\usepackage[normalem]{ulem}

\usepackage[inline]{enumitem}
\setlist[enumerate,1]{label=\arabic*)} 
\usepackage{verbatim}
\usepackage{url}
\usepackage{xcolor}
\usepackage[notref,notcite,color]{showkeys}
\usepackage{microtype}
\usepackage{varwidth} 
\usepackage{mathtools}
\usepackage{mathdots}
\usepackage[all]{xy}
\usepackage{tikz}
\makeatletter
\global\let\tikz@ensure@dollar@catcode=\relax
\makeatother
\usetikzlibrary{fit,positioning,arrows,calc,positioning}

\usepackage[LogicalSystemsSans]{ak_logic}

\newcommand*{\Theorem}{Theorem}
\newcommand*{\Proposition}{Proposition}
\newcommand*{\Lemma}{Lemma}
\newcommand*{\Corollary}{Corollary}
\newcommand*{\Definition}{Definition}
\newcommand*{\Question}{Question}
\newcommand*{\Remark}{Remark}
\newcommand*{\Notation}{Notation}

\theoremstyle{plain}
\newtheorem{theorem}{\Theorem}
\newtheorem{proposition}[theorem]{\Proposition}
\newtheorem{corollary}[theorem]{\Corollary}
\newtheorem{lemma}[theorem]{\Lemma}
\theoremstyle{definition}
\newtheorem{definition}[theorem]{\Definition}
\newtheorem{question}[theorem]{\Question}
\theoremstyle{remark}

\usepackage{prettyref}
\newrefformat{thm}{\Theorem~\ref{#1}}
\newrefformat{pro}{\Proposition~\ref{#1}}
\newrefformat{cor}{\Corollary~\ref{#1}}
\newrefformat{lem}{\Lemma~\ref{#1}}
\newrefformat{rem}{\Remark~\ref{#1}}
\newrefformat{def}{\Definition~\ref{#1}}
\newrefformat{fig}{Figure~\ref{#1}}

\newcommand{\bme}{\lp{BME_*}}
\def\sqsubsetneq{\mathrel{\sqsubseteq\kern-0.92em\raise-0.15em\hbox{\rotatebox{313}{\scalebox{1.1}[0.75]{\(\shortmid\)}}}\scalebox{0.3}[1]{\ }}}
\def\sqsupsetneq{\mathrel{\sqsupseteq\kern-0.92em\raise-0.15em\hbox{\rotatebox{313}{\scalebox{1.1}[0.75]{\(\shortmid\)}}}\scalebox{0.3}[1]{\ }}}

\title[On principles between $I\Sigma_1$ and $I\Sigma_2$, and monotone enumerations]{On principles between $\Sigma_1$- and $\Sigma_2$-induction, \\ and monotone enumerations}

\author{Alexander P.\ Kreuzer}
\address{Department of Mathematics \\
Faculty of Science \\
National University of Singapore \\
Block S17, 10 Lower Kent Ridge Road \\
Singapore 119076 
}
\email{matkaps@nus.edu.sg}
\urladdr{\url{http://aleph.one/matkaps/}}
\thanks{The first author is grateful to Leszek Kolodziejczyk for remarks to an earlier version of this paper. He was supported by the Ministry of Education of Singapore through grant R146-000-184-112 (MOE2013-T2-1-062).}

\author{Keita Yokoyama}
\address{School of Information Science \\
  Japan Advanced Institute of Science and Technology \\
  1-1 Asahidai, Nomi, Ishikawa, 923-1292, Japan
}
\email{y-keita@jaist.ac.jp}
\urladdr{\url{http://www.jaist.ac.jp/~y-keita/}}
\thanks{The work of the second author is partially supported by JSPS Grant-in-Aid for Research Activity Start-up grant number 25887026, JSPS fellowship for research abroad, and JSPS Core-to-Core Program (A. Advanced Research Networks).}

\thanks{Part of this work in this paper was done at the Institute of Mathematical Sciences (IMS) at National University of Singapore during the workshop ``Sets and Computations''.}

\usepackage{scrtime}
\date{\today\ \thistime}

\begin{document}

\begin{abstract}
  We show that many principles of first-order arithmetic, previously only known to lie strictly between $\Sigma_1$-induction and $\Sigma_2$-induction, are equivalent to the well-foundedness of $\omega^\omega$. 
  Among these principles are the iteration of partial functions ($P\Sigma_1$) of H\'{a}jek and Paris, the bounded monotone enumerations principle (non-iterated, \lp{BME_1}) by Chong, Slaman, and Yang, the relativized Paris-Harrington principle for pairs, and the totality of the relativized Ackermann-P\'{e}ter function.
  With this we show that the well-foundedness of $\omega^\omega$ is a far more widespread than usually suspected.

Further, we investigate the $k$-iterated version of the bounded monotone iterations principle (\lp{BME_{\mathnormal{k}}}), and show that it is equivalent to the well-foundedness of the $k+1$-height $\omega$-tower $\omega^{\iddots^\omega}$.
\end{abstract}

\maketitle

In this paper we will investigate principles between $\Sigma_1$\nobreakdash-induction ($I\Sigma_1$) and $\Sigma_2$\nobreakdash-induction ($I\Sigma_2$). The following principles will be considered.
\begin{enumerate}
\item Iteration of partial functions, as introduced by H\'{a}jek, Paris in \cite{HP86}.
\item The bounded monotone enumeration principle (non-iterated), as introduced by Chong, Slaman, Yang in their proof of the fact that Ramsey's theorem for pairs and two colors (\lp{RT^2_2}) does not imply $\Sigma_2$-induction in \cite{CSY14,CSYtaInd}.
\item The relativized Paris-Harrington principle for pairs and arbitrarily many colors.
\item The totality of the Ackermann-P\'{e}ter function relativized to a total function.
\item The well-foundedness of $\omega^\omega$ (\lp{WF(\omega^\omega)}).
\end{enumerate}
Of all of these principles it is well known that they lie strictly between $I\Sigma_1$ and $I\Sigma_2$. However, their relations were mostly unknown. To the knowledge of the authors it was only known that the well-foundedness of $\omega^\omega$ implies the totality of the Ackermann-P\'{e}ter function, and that this is equivalent to the (non-relativized) Paris-Harrington principle for pairs.

We will show that all of the above-enumerated principles are equivalent over $I\Sigma_1$. 
This is surprising since these principles usually have been investigated separately, and the connection was apparently not expected.
For instance in \cite{HP98} the Pairs-Harrington principle, iteration of partial functions, and \lp{WF(\omega^\omega)} are considered but in separate sections.
The system \lp{WF(\omega^\omega)} has shown up in even more places before.
In \cite{sS88} Simpson showed that it is equivalent to Hilbert's basis theorem. Recently, Hatzikiriakou and Simpson that also a related result by Formanek and Lawrence on group algebras is equivalent to, see \cite{HSta}.

Given the many equivalent forms of \lp{WF(\omega^\omega)} of which many are natural statements, we believe that \lp{WF(\omega^\omega)} must be considered as a natural and robust system just like $B\Sigma_2$, which for instance occurs in the natural description as the infinite pigeonhole principle or as a certain partition principle, see \cite{CLY10}.

In addition to this we also investigate $k$-iterated bounded monotone enumeration principle as used in \cite{CSY14}, and characterize its strength. 
We will show that the $k$\nobreakdash-iterated version \lp{BME_\mathnormal{j}} is equivalent to the well-foundedness of $k+1$\nobreakdash-high $\omega$\nobreakdash-tower $\omega^{{\iddots^\omega}}=\omega^\omega_k$.
In particular, the $\Pi^0_3$-consequence of $\lp{BME} = \bigcup_{k\in \Nat} \lp{BME_k}$ are all  $\Pi^0_3$-sentences of \ls{PA}.

The paper is structured as follows. The first chapter will  introduce the principles mentioned above. In the following chapter the equivalences between them are proven. The third chapter deals with the iterated bounded monotone enumeration. The last chapter consists of concluding remarks.

\section{Introduction}

We will work over $I\Sigma_1$, that is Peano Arithmetic where the induction axiom is restricted to $\Sigma_1$-formulas. We will make use of stronger forms of induction (i.e., $I\Sigma_n$ with $n\ge2 $) and the bounded collection principle (i.e., $B\Sigma_n$).
If the reader is not familiar with these systems and principles, we refer him to \cite{HP98}.

\subsection{Iteration of functions}\label{sec:func}
A formula $\phi(x,y)$ represents a total function if $\Forall{x}\ExistsUn{y} \phi(x,y)$, it represents a partial function if for all $x$ there is at most one $y$ satisfying $\phi(x,y)$. We shall denote these statements by $\textsf{TFUN}(\phi)$, respectivly $\textsf{PFUN}(\phi)$.
We shall say that $s$ is an approximation to the iteration of such a function, if $s$ is a finite sequence such that
\[
\Forall{i< \lth(s)\!-\!1}\,\Forall{x,y} \left(\left(x \le (s)_i \AND \phi(x,y)\right) \IMPL y \le (s)_{i+1}\right)
.\]
We will denote this statement by $\textsf{Approx}_\phi(s)$.
The statement that all finite approximations of the iterations of a total resp.\ partial function is given by $\phi$ is then given by the following.
\begin{align*}
  (T\phi)\colon &\qquad \textsf{TFUN}(\phi) \IMPL \Forall{z}\Exists{s} \textsf{Approx}_\phi(s) \AND \lth(s)=z \\
  (P\phi)\colon &\qquad \textsf{PFUN}(\phi) \IMPL \Forall{z}\Exists{s} \textsf{Approx}_\phi(s) \AND \lth(s)=z
\end{align*}
These definitions are made relative to $I\Sigma_1$. For a class of formulas $\mathcal{K}$, the sets $\{ T\phi \mid \phi \in \mathcal{K}\} \cup I\Sigma_1$, $\{ P\phi \mid \phi \in \mathcal{K}\} \cup I\Sigma_1$ will be denoted by $T\mathcal{K}$ resp.\ $P\mathcal{K}$.

The following theorem collects the known facts about $T$, $P$.
\begin{theorem}[\cite{HP86}, {\cite[Chap.~I.2.(b)]{HP98}}] \mbox{}\label{thm:tpkn}
  \begin{enumerate}
  \item $T\Sigma_{n+1} \IFF T\Pi_n$, $P\Sigma_{n+1} \IFF P\Pi_n$, $P\Sigma_0 \IFF P\Sigma_1$.
  \item $T\Sigma_{n+1} \IFF I\Sigma_{n+1}$.
  \item $I\Sigma_{n+2} \IMPL P\Sigma_{n+1} \IMPL I\Sigma_{n+1}$. Here all implications are strict.
  \item $P\Sigma_{n+1}$ is incomparable with $B\Sigma_{n+2}$.
  \item $P\Sigma_{n+1}+B\Sigma_{n+2}$ is strictly weaker than $I\Sigma_{n+2}$.
  \end{enumerate}
\end{theorem}

\subsection{Bounded monotone iterations}

Bounded monotone iterations will deal with enumerations of trees of natural numbers ($\Nat^{<\Nat}$). 

Let $E$ be a function given by a quantifier-free formula. We will regard $E[s]$ via a suitable coding as a finite subset of  $\Nat^{<\Nat}$ and assume that $E[s] \subseteq E[s+1]$.
We will refer to the parameter $s$ of $E$ as the stage of the enumeration and use $E$ also to refer to the tree enumerated by $E$, i.e.,
\[
\left\{\, \tau \in \Nat^{<\Nat} \sizeMid \Exists{s} \Exists{\sigma\in E[s]} \left(\tau \prec \sigma\right) \,\right\}.
\]

\begin{definition}[\cite{CSY14}]\label{def:me}
  $E$ is a \emph{monotone enumeration} if the following holds.
  \begin{enumerate}
  \item The empty sequence $\langle \rangle$ is enumerated at the first stage.
  \item At each stage only finitely many sequences are enumerated by $E$. (This is by our coding automatically the case.)
  \item\label{enum:me3} If $\tau$ is enumerated by $E$ at stage $s$ and $\tau_0$ is the longest initial segment enumerated by $E$ at a prior stage. Then
    \begin{enumerate}
    \item no extension of $\tau_0$ has been enumerated by $E$ before the stage $s$ and
    \item\label{enum:me3b} all sequences enumerated at stage $s$ are extensions of $\tau_0$.     
    \end{enumerate}
  \end{enumerate}
\end{definition}

Let $E$ be a monotone enumeration. For an element $\tau$ enumerated by $E$ at stage $s$ we call the maximal initial segments $(\tau_i)$ of $\tau$ enumerated at stages prior to $s$ the stage-by-stage sequence of $\tau$.
\begin{figure}
  \centering
  \begin{tikzpicture}
    [-, grow'=up,
    level/.style={level distance=1cm},
    level 1/.style={sibling distance = 4cm },
    level 2/.style={sibling distance = 1.3cm },
    level 3/.style={sibling distance = 1cm },
    ]
    \node (nr)  {$\langle\rangle$}  
    child { 
      node (n1) {$\tau_1$}
      child { node (n11) {$\tau_{11}$} }
      child { node (n12) {$\tau_{12}$}  
        child { node (n121) {$\tau_{121}$}  }
        child { node (n122) {$\tau_{122}$}  }
      }
      child { node (n13) {$\tau_{13}$}  }
      child { node (n14) {$\tau_{14}$}  
        child { node (n141) {} }
        child { node (n142) {}} 
      }
    }
    child{
      node (n2) {$\tau_2$}
      edge from parent
      child {node (n21) {$\tau_{21}$} edge from parent}
    };

    \path (n141) -- (n142) node[midway] {$\vdots$};

    \node (s1) [fit=(n1) (n2), draw=gray]{};
    \node (s11) [fit=(n11) (n12) (n13) (n14), draw=gray]{};
    \node (s112) [fit=(n121) (n122), draw=gray]{};
    \node (s2) [fit=(n21), draw=gray]{};      
  \end{tikzpicture}

  \small
  Strings $\tau_i$ in a box are enumerated at the same stage. 
  The stage-by-stage enumeration of say $\tau_{121}$ is $\langle\rangle, \tau_1,\tau_{12},\tau_{121}$.
  Not visible in the diagram is that notes enumerate at the same stage, say $\tau_1$, $\tau_2$, might be of different length.
  \caption{Tree enumerated by a monotone enumeration.}
  \label{fig:monenu}
\end{figure}
We say that \emph{a monotone enumeration $E$ is bounded by $b$} if for each $\tau$ in $E$ the length of its stage-by-stage sequence is bounded by $b$.

\begin{definition}
  \bme is the statement that a tree enumerated by a bounded monotone enumeration is finite.
\end{definition}

The following is known about the first-order strength of \bme.
\begin{theorem}[{\cite[Propositions~3.5, 3.6]{CSY14}}]\label{thm:bmeknow} \mbox{}
  \begin{enumerate}
  \item $I\Sigma_2 \vdash \bme$
  \item $B\Sigma_2 \nvdash \bme$
  \end{enumerate}
\end{theorem}

Note that \bme is equivalent to \lp{BME_1} as defined by Chong, Slaman an Yang. This follows for instance from Theorems~\ref{thm:m} and \ref{thm:bme}.

\subsection{Paris-Harrington theorem}

The Paris-Harrington theorem (\lp{PH}) is a strengthening of the finite Ramsey's theorem.
It is one of the classical examples of a natural first-order theorem which is not provable from Peano Arithmetic. 
In this paper we will be only concerned with a (variant of a) fragment of \lp{PH}.

As usual in this context, we will write $X \rightarrow (q)^u_z$ for the statement that each coloring of unordered $u$-tuples of $X$ with $z$ colors has a homogenous set of cardinality $q$. 
In this notation finite Ramsey's theorem is simply the statement 
\[
\Forall{q \ge 1} \Forall{u\ge 1} \Forall{z} \Exists{y} \big([0,y] \rightarrow (q)^u_z\big)
.\]

To state the Paris-Harrington variant of Ramsey's theorem we will need the following.
A finite set $X$ is called \emph{relatively large} if $\min X < \lvert X \rvert$. 

We will write $X \underset{*}\rightarrow (q)^u_z$ if each coloring of unordered $u$-tuples of $X$ with $z$ colors has a \emph{relatively large} homogenous set of cardinality at least $q$. The Paris-Harrington theorem is then the following statement.
\[
(\lp{PH})\colon \Forall{x} \Forall{q \ge 1} \Forall{u\ge 1} \Forall{z} \Exists{y} \big([x,y] \underset{*}\rightarrow (q)^u_z\big).
\]
(Note that we need to vary the starting point $x$ of the interval since the property of being relatively large is not translation invariant.)

We will write $\lp{PH(\mathnormal{u},\mathnormal{z})}$ for the restriction of $\lp{PH}$ to $u$-tuples and $z$ many colors. We will write $\lp{PH(\mathnormal{u})}$ for $\Forall{z} \lp{PH}(u,z)$.

We will also need the relativization \lp{PH^*(\mathnormal{u},\mathnormal{z})} of \lp{PH(\mathnormal{u},\mathnormal{z})} given by the following. 
Let $\phi(n)$ be a $\Sigma_1$-formula describing an infinite set. Then  \lp{PH^*(\mathnormal{u},\mathnormal{z})} states that $\lp{PH(\mathnormal{u},\mathnormal{z})}$ holds relativized to $[x,y] \cap \{n\mid \phi(n)\}$. In other words,
\[
\lp{PH^*(\mathnormal{u},\mathnormal{z})}: \Forall{k}\Exists{n>k} \phi(n) \IMPL \Forall{x}\Forall{q\ge 1} \Exists{y} \Big(\big([x,y]\cap \{n\mid \phi(n)\}\big) \underset{*}{\rightarrow} (q)^u_z\Big)
.\]
\lp{PH^*(\mathnormal{u})} is defined as above.

We will be mainly concerned with \lp{PH(2)}, \lp{PH^*(2)}.

\subsection{Ackermann function}

The Ackermann-P\'{e}ter function is given by the following defining equations.
\begin{equation}\label{eq:ack}
A(m,n) :=
\begin{cases}
  n + 1 & \text{if $m=0$,} \\
  A(m-1,1) & \text{if $m>0$ and $n=0$,} \\
  A(m-1, A(m,n-1)) & \text{if $m,n>0$,}
\end{cases}
\end{equation}
It is known that $I\Sigma_2$ or even the statement that $\omega^\omega$ is well-order implies the totality of Ackermann-P\'{e}ter function.
Let $f$ be a strictly monotonic function.
The relativized Ackermann-P\'{e}ter function $A_f$ is defined as $A$ but with the base case set to $f$, i.e., 
\begin{equation}\label{eq:ackf}
  A_f(0,n) := f(n).
\end{equation}
We will write $\lp{A^*}$ for the statement that for each function $f$ (given by a quantifier-free formula) the Ackermann-P\'{e}ter function relative to $f$ is total.

\subsection{Ordinals}

We will use ordinals $<\epsilon_0$. For this we will fix a suitable ordinal notation. See e.g.\ \cite[Section~II.3]{HP98} for details.
We shall write \lp{WF(\alpha)} for the statement that $\alpha$ if \emph{well-ordered} or \emph{well-founded},
that is there is no infinite descending sequence of ordinals $\alpha_i$ starting from $\alpha$.
In the context of fragments of first-order arithmetic the sequence $\alpha_i$ is understood to be primitive recursive in the theory, which is equivalent to saying that $\alpha_i$ is given as a $\Sigma_1$-function, as defined in \prettyref{sec:func}.

Since our work is motivated by results in second-order arithmetic/reverse mathematics, we would note that in that context well-foundedness is defined differently, see \cite{sS88}.
There the descending sequence $\alpha_i$ is given by a second-order  object $X$ coding the function $f\colon i \mapsto \alpha_i$.
Since that $\Sigma_1$-function in the sense of \prettyref{sec:func} are exactly the functions from which a theory proves to be recursive, recursive comprehension gives that the $\Sigma_1$-functions and the second-order functions coincide.
This immediately shows that the first- and second-order definitions of well-foundedness are equivalent.

\medskip

The main result of this paper is the following.
\begin{theorem}\label{thm:m}
  Over $I\Sigma_1$ the following are equivalent:
  \begin{enumerate}[label=(\roman*)]
  \item\label{enum:m1} $P\Sigma_1$,
  \item\label{enum:m2} \lp{BME_*},
  \item\label{enum:m3} \lp{PH^*(2)},
  \item\label{enum:m4} \lp{A^*},
  \item\label{enum:m5} \lp{WF(\omega^\omega)}.
  \end{enumerate}
\end{theorem}

The proof will proceed as follows.  
\begin{itemize}
\item $\ref{enum:m1} \Leftrightarrow \ref{enum:m2}$ (\prettyref{pro:psigbme}),
\item $\ref{enum:m1} \Rightarrow \ref{enum:m4}$ (\prettyref{pro:pack}),
\item $\ref{enum:m4} \Rightarrow \ref{enum:m3}$ (\prettyref{pro:PHack}),
\item $\ref{enum:m3} \Rightarrow \ref{enum:m5}$ (\prettyref{pro:ph2om})
\item $\ref{enum:m5} \Rightarrow \ref{enum:m2}$ (\prettyref{pro:bme})
\item $\ref{enum:m5} \Rightarrow \ref{enum:m4}$ is a classical result.
\end{itemize}

\section{The proof of \prettyref{thm:m}}

\begin{proposition}\label{pro:psigbme}
  $I\Sigma_1 \vdash P\Sigma_1 \IFF \bme$
\end{proposition}
\begin{proof}
  ``$\rightarrow$'': Let $E$ be a monotone enumeration. Assume that $E$ is bounded by $b$.
  Define the partial functions
  \begin{align*}
    F'(\tau) &:= [\text{first stage $s$ such that extensions of $\tau$ are enumerated in $E$}]
    \shortintertext{and}
    F(\tau) &:= E[F'(\tau)]\setminus E[F'(\tau)-1] .
  \end{align*}
  The partial function $F(\tau)$ yields the set of all extensions of $\tau$ that are newly enumerated at the first stage where extensions of $\tau$ enter into $E$. Since $E$ is a monotone enumeration these are all direct extensions of $\tau$.

  The graph of $F'$ can be defined by the following $\Sigma_0$\nobreakdash-formula 
  \[
  \phi'(\tau,s):\equiv  \Exists{\tau'\in E[s]\setminus E[s-1]} \,\left(\tau\prec\tau'\right) \AND \Forall{\tau'\in E[s-1]}\, \left(\tau\nprec \tau'\right)
  .\] 
  The partial function $F$ can then be defined by the $\Sigma_1$\nobreakdash-formula 
  \[
  \phi(\tau,x) :\equiv \Exists{s} \left(\phi'(\tau,s) \AND x=\big[E[s]\setminus E[s-1]\big]\right)
  .\]

  We make the assumption that for each code of a finite set $x$ we have that $y\in x$ implies $y\le  x$. (This is for instance the case for the usual coding based on Cantor pairing.)

  Then we have for each stage-by-stage enumeration $(\tau_i)$ that $\tau_{i+1} \le F(\tau_i)$. Hence $\tau_{i+1} \le F^{i+1}(\tau_0)=F^{i+1}(\langle \rangle)$.
  As a consequence each element in any $b$\nobreakdash-bounded stage-by-stage enumeration is bounded by $\max_{i\le b} \{ F^i(\langle\rangle)\} $. Now by $P\Sigma_1$ we can bound this value and obtain that $E$ is finite.

  \medskip
  ``$\leftarrow$'': Let $\phi(x,y)$ be a quantifier-free formula and assume $\textsf{PFUN}(\phi)$. (Quantifier-free is sufficient by \prettyref{thm:tpkn}.(1).) 
  Let $b$ be given. 
  We will construct a $b$\nobreakdash-bounded monotone enumeration $E$ which will give an approximation $s$ of length $b$ to the iteration of $\phi$. \\
  At stage $0$ we will enumerate $\langle 0 \rangle$ into the tree.\\
  At stage $s+1$ we search for the smallest $\sigma=\langle x_0,\dots,x_k\rangle\in E[s]$ such that $\lvert \sigma \rvert < b$ and $\Exists{y<s+1} \phi(x_k,y)$.
  If such a $\sigma$ exists then enumerate $\sigma \ast \langle 0 \rangle, \sigma \ast \langle 1 \rangle, \dots, \sigma \ast \langle y \rangle$. Otherwise do nothing.

  By $\bme$ this tree is finite.  
  Let $m_i$ be the maximum of the elements in the $\le i$ levels of $E$. We claim that $s=\langle m_0,m_1,\dots,m_{b}\rangle$ satisfies $\textsf{Approx}_\phi(s)$.
  We prove this by induction on the length of $s$. For $\langle m_0 \rangle = \langle 0 \rangle$ this is clear.
  Assume that the statement is true for $\langle m_0,\dots, m_i\rangle$. First we consider the case that the maximum $m_{i+1}$ is attained at a level $<i$, i.e., $m_i=m_{i+1}$ and by the construction of $E$ we have that $\Forall{x \le m_i}\Forall{y} \phi(x,y) \IMPL y\le m_i$.
  From this it follows immediately that also $\langle m_0,\dots, m_i,m_i\rangle$ satisfies $\textsf{Approx}_\phi$.
  Now consider the case that $m_{i+1}$ is attained at the $(i+1)$\nobreakdash-th level and no prior level. By construction of $E$ there must be the elements $[0;m_{i}]$ on the $i$\nobreakdash-th level, and we have $\Forall{x<m_{i}} \Forall{y} \phi(x,y) \IMPL y\le m_{i+1}$, which yields that $\langle m_0,\dots, m_i,m_{i+1}\rangle$ satisfies $\textsf{Approx}_\phi$.
\end{proof}

\begin{proposition}\label{pro:pack}
  $I\Sigma_1 \vdash P\Sigma_1  \IMPL \lp{A^*}$.
\end{proposition}
\begin{proof}
  For notational ease we will only show that $A(m,n)$ is total. The relativization to $A_f(m,n)$ is straightforward.

  Let $\phi_A(m,n,k)$ be the $\Sigma_1$\nobreakdash-formula describing the graph of the (relativized) Ackermann-P\'{e}ter function $A$ as in \eqref{eq:ack} and $\psi_A(m,n)\equiv \Exists{k} \phi_A(m,n,k)$ be the $\Sigma_1$\nobreakdash-formula which states that $A(m,n)$ is defined.
  Clearly,
  \begin{equation}\label{eq:ack1}
    \Forall{n}\, \psi_A(0,n).
  \end{equation}

  We claim that $I\Sigma_1$ proves
  \begin{equation}\label{eq:ack2}
    \Forall{m,n} \left( \NOT \psi_A(m,n)\IMPL \Exists{n'} \NOT \psi_A(m-1,n')\right). 
  \end{equation}
  Indeed, suppose  $\NOT \psi_A(m,n)$ and in particular that $m>0$. Then by $I\Sigma_1$ we can find a $k$ which is minimal with $\NOT \psi_A(m,k)$. If $k=0$ then by definition of $A$ we have $\NOT \psi_A(m-1,1)$. If $k>0$ then by minimality $A(m,k-1)$ is defined, thus $A(m-1,A(m,k-1))$ cannot be defined and therefore $\NOT \psi_A(m-1, A(m,k-1))$.

  $\Sigma_2$-induction applied to \eqref{eq:ack2} would now immediately give that $\NOT \psi_A(m,n)$ implies $\Exists{n'} \NOT \psi_A(0,n')$. ($\Sigma_2$-induction is required since $\Exists{n'} \NOT \psi_A(m,n')$ is $\Sigma_2$.)
  Together with \eqref{eq:ack1} this would yield the totality of $A$.

  We will show how to use $P\Sigma_1$ to bound $n'$ occurring in \eqref{eq:ack2}. With this, $I\Sigma_1$ suffices to carry out this induction.

  Let $\langle m,n \rangle$ denote the Cantor pairing function and $(x)_0,(x)_1$ the unpairing functions. Recall that $m,n < \langle m,n \rangle$. To cover both parameters of $A(m,n)$ we will use the following modification
  \[
  A'(x) := \left\langle A\big((x)_0,(x)_1\big), A\big((x)_0,(x)_1\big) \right\rangle
  \]
  Let $\phi_{A'}(x,k)$ be the $\Sigma_1$-formula describing the graph of $A'$.

  Suppose that $A(m,n)$ is not defined or in other words $\NOT \psi_A(m,n)$. Let $c:=\max(m,n)$.

  Now by $P\Sigma_1$ arbitrary long approximations to $A'$ exists. 
  Since $A(0,n)=n+1$, and assuming that $\langle 0, 0 \rangle = 0$, which is the case for Cantor pairing, we have for any approximation $s$ of $A'$
  \[
  (s)_j \ge \langle j+1,j+1 \rangle, \qquad \text{for }j<\lth(s) .
  \]
  Therefore, if $A(m,n)$ with $m,n<c $ is defined then $A(m,n) \le (s)_c$ 
  for any approximation $s$ to $A'$ of length $> c$.

  Now as in the argument above, assume that $A(m,n)$ is not defined. Then we know that there is a $k<m$ such that $A(m,k-1)$ is defined and $A(m-1,A(m,k-1))$ is not defined or $A(m-1,1)$ is not defined. In particular, for a long enough approximation $s$ of $A'$ we have
  \[
  \Exists{n' < (s)_c} \NOT \phi_A(m-1,n')
  .\]
  Since $m,n'$ are bound by $(s)_c$ one obtains by the same argument that
  \[
  \Exists{n'' < (s)_{c+1}} \NOT \phi_A(m-2,n'')
  .\]
  Iterating this argument gives then 
  \[
  \Exists{n^* < (s)_{c+m-1}} \NOT \phi_A(0,n^*)
  \]
  and with this the desired contradiction to \eqref{eq:ack1}. This argument can be carried out in $P\Sigma_1$ since this iteration is---after building the approximation $s$ of sufficient ($=2c$) length---provable in $I\Sigma_1$ which is a consequence of $P\Sigma_1$.
\end{proof}

It is known that the totality of the Ackermann function implies \lp{PH}, see Theorem~II.3.36 and Fact~II.3.34 of \cite{HP98}. We show here how to relativize this proof to obtain the following theorem.
\begin{proposition}\label{pro:PHack}
  $I\Sigma_1\vdash \lp{A^*} \IMPL \lp{PH^*(2)}$.
\end{proposition}

Before we can prove this theorem we will need some notation and lemmata. In a canonical way we can define a fundamental sequence $\{\alpha\}(n)$ for each $\alpha< \epsilon_0$. That is a sequence such that $\{\alpha\}(n)$ converges monotonically from below to $\alpha$ if $\alpha$ is a limit and the predecessor otherwise. For instance $\{\omega\}(n) = n$. This sequence will be $\Delta_1$. See \cite[II.3.a)]{HP98} for details.

We say that a finite set $X=\{x_0<x_1<x_2<x_3<\dots<x_n\}$ is $\alpha$-large if the sequence
\[
\{\alpha\}(x_0), \big\{\{\alpha(x_0\}\big\}(x_1), \big\{\big\{\{\alpha(x_0\}\big\}(x_1)\big\}(x_2), \dots 
\]
reaches $0$. It is easy to see that $\omega$-large is the same as relatively large (by using the fact $\{\omega\}(n)=n$ and $\{n\}(m)=n-1$).

\begin{lemma}[{\cite[Section~6.2]{KS81}}]\label{lem:KS81}
  Let $z\ge 2$, $\theta:=\omega^{z+3}+\omega^3+z+4$. Further, let $X$ be an $\theta$-large set. Assume that the pairs of $X$ are colored with $z$ many colors. There exists a subset $Y$ of $X$ that is homogenous and relatively large.

  In other words, for $X$ we have that the conclusion of $\lp{PH^*}(2,z)$ holds.
\end{lemma}

\begin{lemma}[{\cite[Lemma~II.3.21.(3)]{HP98}}]\label{lem:hp98-321}
  Suppose $\alpha \gg \beta > 0$ (that means, looking at the Cantor-normals forms of $\alpha=\sum_{i=0}^x \omega^{\mu_i} a_i$, $\beta=\sum_{i=0}^y \omega^{\nu_i} b_i$ we have that $\mu_0 \ge \nu_y$). Then $X$ is $(\alpha+\beta)$-large if{f} there are $X_\alpha,X_\beta$ such that $X=X_\beta\cup X_\alpha$, $\max(X_\beta) < \min(X_\alpha)$, and $X_\alpha$ is $\alpha$-large and $X_\beta$ is $\beta$-large.  
\end{lemma}

\begin{lemma}[cf.~{\cite[Lemma~II.3.30.(3)]{HP98}}]\label{lem:hp98-330}
  Let $g$ be the strictly increasing enumeration of an infinite set $X$. Let $f_\alpha$ be the fast growing hierarchy relativized to $g$ as follows.
  \begin{equation}\label{eq:hp98-330}
    \begin{split}
      f_0(n) &:= g(n) \\
      f_{\beta+1}(n) &:= f^{n}_\beta(g(n+1)), \quad\text{where $f^n$ is the $n$-fold iteration}\\
      f_\lambda(n) &:= f_{\{\lambda\}(n)}(g(n+1)).
    \end{split}
  \end{equation}
  If $x\in X$, the set $[x,f_{\alpha}(x)] \cap X$ is $\omega^\alpha$-large.
\end{lemma}
\begin{proof}[Proof of \prettyref{lem:hp98-330}]
  First observe that for all $\alpha,n$ we have $f_\alpha(n)\in X$. We will use the following claim.
    
  \noindent \textbf{Claim: }
  Assume that the statement of the lemma holds for $\alpha$ and that $x\in X$. Then the set  $[x,f_{\alpha}^y(x)] \cap X$ is $\omega^\alpha \cdot y$-large. \\
  \textbf{Proof of claim: }
  The statement is shown by induction in $y$. Suppose $[x,f_{\alpha}^y(x)] \cap X$ is $\omega^\alpha \cdot y$ large. By the assumption we have that $[x, f_{\alpha}(x)] \cap X$ is $\omega^\alpha$-large, and by induction hypothesis that $[f_\alpha(x),f_{\alpha}^y(f_{\alpha}(x))]\cap X$ is $\omega^\alpha \cdot y$-large. Now \prettyref{lem:hp98-321} gives the claim.
  \medskip
  
  We prove the lemma by quantifier-free transfinite induction. (We will use it only for $\alpha<\omega$ in the proof of \prettyref{pro:PHack}.)
  Consider $\alpha+1$ and $x=g(n)\in X$. Now $[x,z]\cap X$ is $\omega^{\alpha+1}$-large if{f} $[g(n+1),z]\cap X$ is $\omega^{\alpha} \cdot x$-large, i.e., if $z\ge f_{\alpha}^x(g(n+1))$. Since $f_{\alpha}^x(g(n+1)) \le f_{\alpha}^{x}(g(x+1)) = f_{\alpha+1}(x)$, the claim follows.
  For the limit case consider $\lambda$ and again $x=g(n)\in X$. Then $[x,z] \cap X$ is $\omega^\lambda$-large if{f} $[g(n+1),z] \cap X$ is $\omega^{\{\lambda\}(x)}$-large, i.e., $z\ge f_{\{\lambda\}(x)}(g(n+1))$. Thus it suffices if $z\ge f_{\{\lambda\}(x)}(g(x+1))=f_\lambda(x)$.
\end{proof}

\begin{proof}[Proof of \prettyref{pro:PHack}]
  Let $\phi(n)$ be a $\Sigma_1$-formula describing an infinite set. Assume that a number of colors $z$ is given. By \prettyref{lem:KS81} (we check that it formalizes in $I\Sigma_1$) it is sufficient to find a $\theta$-large subset of $X:=\{n\mid \phi(n)\}$. 
  We can apply \prettyref{lem:hp98-330} to $X$ (a suitable $g$ exists by $I\Sigma_1$) and reduce the problem to showing that $f_{z+4}(x)$ as in \eqref{eq:hp98-330} is total. This follows from the totality of the relativized Ackermann-P\'{e}ter function. (We have for instance that $A_g(2k,n)$ majorizes $f_k(n)$.)
\end{proof}

\begin{proposition}\label{pro:ph2om}
  $I\Sigma_1\vdash \lp{PH^*(2)} \IMPL \lp{WF(\omega^\omega)}$.
\end{proposition}
\begin{proof}
  It is well known that the order of $\omega^\omega$ is isomorphic to the lexicographic order $<^*$ of $\Nat^{<\Nat}$. (To see this consider the order-isomorphism $n_0n_1n_2\cdots n_k \mapsto \omega^k \cdot (n_k+1) + \dots + \omega^2 \cdot n_2 + \omega^1 \cdot n_1 + n_0$.)
  
  Assume that $\omega^\omega$ is not well-ordered. Then there is a function $f\colon\Nat \longto \Nat^{<\Nat}$ such that $f(n) \mathrel{{}^*{>}} f(n+1)$. We will show that this contradicts \lp{PH^*(2)}.
  Let $b:=\lth(f(0))$. By definition of the lexicographic order we know that $\lth(f(n))\le b$ for all $n$. We define a $\Delta_1$-set $X$ and a strictly increasing $\Delta_1$-function $h\colon X\longto \Nat$, such that $\max_i \big(f(h(n))\big)_i < n$ and $\min(X)>b$. Such $X,h$ can be build by primitive recursion by
  \begin{align*}
  h(0) &:=0, \\
  h(n+1) &:=
  \begin{cases}
    h(n)+1 & \text{if }\max_i \big(f(h(n)+1)\big)_i < n+1, \\
    h(n) & \text{otherwise.}
  \end{cases} \\
  X &:= \{ n> \max\nolimits_i\big(f(0)\big)_i, b \mid h(n) \neq h(n-1) \} 
  .\end{align*}
  It is clear that $X$ is infinite.

  Define the coloring $c\colon [X]^2 \longto b\cup \{-1\}$ by the following
  \[
  c(\{n,m\}):= 
  \begin{cases}
    \max\left(\left\{\, i<b \sizeMid 
        \begin{multlined}
          \big(f(h(n))\big)_i \neq \big(f(h(m))\big)_i \\
          \AND i<\lth(f(h(m))
        \end{multlined}
      \right\}\right)
    & \text{if such an $i$ exists,} \\
    -1 & \text{otherwise.}
  \end{cases}
  \]
  By \lp{PH^*(2)} there exists a $c$-homogenous, relatively large set $Y\subseteq X$. 
  First assume that $c([Y]^2)=-1$.
  This implies that for $n,m\in Y$ we have 
  \[
  n < m \IMPL  f(h(n)) \sqsupsetneq f(h(m)).
  \]
  Therefore, $\lth(f(h(n)) > \lth(f(h(m))$. Since the length of $f(n)$ is bounded by $b$, 
  there must be a strictly decreasing sequence of natural numbers $\le b$ of length $\lvert Y \rvert > \min Y > b$, which is a contradiction.

  Now assume $c([Y]^2)=i\neq -1$. Then for $n,m\in Y$ we have
  \[
  n<m \IMPL \big(f(h(n))\big)_i > \big(f(h(m))\big)_i.
  \]
  Since $\big(f(h(\min Y))\big)_i < \min Y$, we have decreasing sequence of length $\lvert Y \rvert > \min Y$ of natural numbers $<\min Y$, which is again a contradiction.
\end{proof}

\begin{proposition}\label{pro:bme}
  $I\Sigma_1\vdash \lp{WF(\omega^\omega)} \IMPL \bme$
\end{proposition}
\begin{proof}
  Let $E[s]$ be a $b$-bounded monotone enumeration. We will assign to the trees $E$ and $E[s]$ an ordinal in the following way.

  For $\tau\in E$ let ${\lvert \tau \rvert}_E$ be length of the stage-by-stage enumeration of $\tau$. 
  We say a $\tau$ is maximal in its stage if there is no extension $\tau'\in E$ of $\tau$ with ${\lvert \tau \rvert}_E = {\lvert \tau' \rvert}_E$.
  For maximal $\tau,\tau'\in E$ define $\tau\sqsubset_E \tau'$ if $\tau \sqsubset \tau'$ and ${\lvert\tau\rvert}_E = {\lvert \tau'  \rvert}_E - 1$.
  To a maximal $\tau \in E$ we assign the following ordinal.
  \begin{equation}\label{eq:zeta}
    \zeta_E(\tau) := 
    \begin{cases}
      0 & \text{if ${\lvert \tau \rvert}_E = b$,} \\
      \omega^{b-{\lvert \tau \rvert}_E} & \text{if $\tau$ is a leaf in $E$ and ${\lvert \tau \rvert}_E<b$,} \\
      \sum_{\tau'\sqsupset_E \tau} \zeta_E(\tau') & \text{if $\tau$ is not a leaf.}
    \end{cases}
  \end{equation}
  Same for $E[s]$ instead of $E$.
  We define the ordinal $\zeta_E$, $\zeta_{E[s]}$ for $E$ respectively $E[s]$ to be $\zeta_E(\langle\rangle)$, $\zeta_{E[s]}(\langle\rangle)$.

  By definition it is clear that $\zeta_E, \zeta_{E[s]} \le \omega^{b}$. Moreover, we claim that if new elements are enumerated into $E[s+1]$ then
  $\zeta_{E[s+1]} < \zeta_{E[s]}$. Indeed, if there are new elements enumerated at stage $s+1$ there must be a leaf $\tau\in E[s]$ such that all elements are successors of $\tau$. Then by definition we have
  $\zeta_{E[s+1]}(\tau) < \zeta_{E[s]}(\tau)$. Induction on ${\lvert \tau \rvert}_E$, gives that $\zeta_{E[s+1]}(\tau') < \zeta_{E[s]}(\tau')$ for all maximal $\tau'\sqsubset \tau$. In particular $\zeta_{E[s+1]} < \zeta_{E[s]}$.

  Now the stages $s_i$ where new elements are enumerated into $E$ gives a decreasing sequence of ordinals $\zeta_{E[s_i]}<\omega^b$. Since $\omega^b<\omega^\omega$ and $\omega^\omega$ is well-founded by assumption, there can be only finitely many stages where new elements are enumerated and thus $E$ is finite.
\end{proof}

Note that Theorem~\ref{thm:m} can be relativizable with set parameters. In the second-order setting with the recursive comprehension, we can replace primitive recursive sequences / $\Sigma_{1}$-definable infinite sets / functions defined by quantifier-free or $\Sigma_{1}$-formulas by sets. Thus, we have the following.
\begin{theorem}\label{thm:m-sec}
  Over $\mathsf{RCA}_{0}$ the following are equivalent:
  \begin{enumerate}[label=(\roman*)]
  \item\label{enum:m1-sec} $P\Sigma^0_1$: $P\phi$ for any $\Sigma^{0}_{1}$-formulas ($\Sigma_{1}$-formulas with set parametes),
  \item\label{enum:m2-sec} \lp{BME_*}: $\Forall{E}(E$ is a monotone enumeration bounded by $b\IMPL E$ is finite$)$,
  \item\label{enum:m3-sec} \lp{PH^*(2)}: $\Forall{X}\Forall{z}\left((\Forall{k}\Exists{n>k} n\in X) \IMPL \Forall{x}\Forall{q\ge 1} \Exists{y} \Big(\big([x,y]\cap X\big) \underset{*}{\rightarrow} (q)^2_z\Big)\right)$,
  \item\label{enum:m4-sec} \lp{A^*}: $\Forall{f}($the Ackermann-P\'{e}ter function relative to $f$ is total$)$,
  \item\label{enum:m5-sec} \lp{WF(\omega^\omega)}: $\neg\Exists{f}(f$ is an infinite descending sequence of ordinals $\alpha_{i}$ starting from $\omega^{\omega})$.
  \end{enumerate}
\end{theorem}

\section{Full \lp{BME}}

Chong, Slaman, Yang actually used certain iterations of the principle \bme in \cite{CSY14} called \lp{BME_\mathnormal{k}} and $\lp{BME}:=\bigcup_k \lp{BME_\mathnormal{k}}$ for the union of all these. 
In these principles,  bounded monotone enumerations will be enumerated relative to a real (in a continuous way). We will write $E(\sigma)$, with $\sigma\in \Nat^{<\Nat}$, for such an enumeration and understand that the stage $s$ will be implicitly given by $s=\lvert \sigma \rvert$. Further, we will compute a bounded tree in a similar fashion, i.e., by a function $V(\tau)$ where $\tau\in \Nat^{<\Nat}$.
Here, we again consider functions $E$ and $V$ defined by $\Sigma_1$-formulas to work within $I\Sigma_{1}$, but one can easily lift-up the following discussion into the second-order setting as same as Theorem~\ref{thm:m-sec}.
\begin{definition}\mbox{}
  \begin{enumerate}
  \item
    Let $E(\sigma)$ be a monotone enumeration as above.
    For a tree enumerated by $V$ a $\sigma\in V$ is called \emph{$E$-expansionary}
    if in $E(\sigma)$ a new element is enumerated a stage $\lvert \sigma \rvert$.
  \item
    A level $\ell$ in a tree $V$ is $E$-expansionary if there is an $n$ such that $\ell$ is minimal with
    for all $\sigma\in V$ with $\lvert\sigma\rvert = \ell$ and 
    there are at least $n$ $E$-expansionary initial segments of $\sigma$.
  \item A \emph{$k$-iterated monotone enumeration} is a sequence $(V_i,E_i)_{1\le i \le k}$ such that
    \begin{enumerate}
    \item each $V_i$ is a relativized recursively bounded tree as above,
    \item each $E_i$ is a relativized monotone enumeration procedure as above,
    \item for each $1\le j <k$, if $\sigma\in V_j$ is $E_j$-expansionary,
      then for each new element $\tau$ enumerated in $E_j(\sigma)$, $V_{j+1}(\tau)$ is a proper $E_{j+1}$-expansionary extension
      of $V_{j+1}(\tau_0)$, where $\tau_0$ is the longest initial segment of $\tau$ that had been enumerated into $E_{j}(\sigma)$ before.
    \end{enumerate}
  \item A \emph{$k$-path} for a $k$-iterated monotone enumeration (as above) is a sequence $(\sigma_i,\tau_i)_{1\le i \le k}$
    such that $\sigma_1\in V_1$, $\tau_1$ is a maximal sequence in $E_1(\sigma_1)$, and for each $1<j\le k$
    we have that $\sigma_j$ is a maximal sequence in $V_j(\tau_{j-1})$ and $\tau_j$ is a maximal sequence in $E_j(\sigma_j)$.
  \item A $k$-iterated monotone enumeration is $b$-bounded if $E_k(\sigma)$ is $b$-bounded for each $\sigma$.
  \item \lp{BME_\mathnormal{k}} is the statement that each bounded $k$-iterated monotone enumeration procedure contains 
    only finitely many $E_1$-expansionary levels in $V_1$.
  \end{enumerate}
\end{definition}

Let $\omega_{0}^\delta:=\delta$ and $\omega_{k+1}^\delta := \omega^{\omega^\delta_{k}}$.
In particular $\omega^\omega_k = \underbrace{ \omega^{\omega^{\iddots^\omega}}}_{\mathclap{\text{$k+1$ many $\omega$}}}$.
We will show the following theorem. 

\begin{theorem}\label{thm:bme} For all $k$
  \[
  I\Sigma_1 \vdash \lp{BME_\mathnormal{k}} \IFF \lp{WF(\omega^\omega_{\mathnormal{k}})}.
  \]
\end{theorem}
\begin{corollary}
  $I\Sigma_1 \vdash \Forall{k} \lp{BME_\mathnormal{k}} \IFF \lp{WF(\epsilon_0)}$.
\end{corollary}

The proof of \prettyref{thm:bme} proceeds by exhibiting a one-to-one correspondence between $k$-iterated monotone enumerations and ordinals $<\omega^\omega_k$.

For the backward direction of the proof we will consider bounded monotone enumerations of $\Nat$ together with a special termination symbol $\bot$.
This will not cause any problems since $\Nat \cup \{\bot\}$ can of course be code into $\Nat$. 
We will extend the assignment of ordinals to bounded monotone enumerations as in \eqref{eq:zeta} to include a case for $\bot$.
\[
\zeta_E(\tau) := 
\begin{cases}
  0 & \text{if $\tau({\lvert \tau \rvert -1}) = \bot$,} \\
  0 & \text{if ${\lvert \tau \rvert}_E = b$,} \\
  \omega^{b-{\lvert \tau \rvert}_E} & \text{if $\tau$ is a leaf in $E$ and ${\lvert \tau \rvert}_E<b$,} \\
  \sum_{\tau'\sqsupset_E \tau} \zeta_E(\tau') & \text{if $\tau$ is not a leaf.}
\end{cases}
\]
  
Now let a $k$-iterated monotone enumeration $(V_i,E_i)_{1\le i \le k}$ be given. 
Further assume that $\langle\sigma_1,\tau_1,\dots,\sigma_k,\tau_k\rangle$ is a $k$-path in $(V_i,E_i)_{1 \le i \le k}$. We assign the following ordinals.
\begin{align*}
  \zeta_{\langle \sigma_1,\tau_1,\dots,\sigma_k \rangle}(\tau) &:= \zeta_{E_k(\sigma_k)}(\tau), \\
  \zeta_{\langle \sigma_1,\tau_1,\dots,\sigma_j,\tau_j \rangle}\phantom{(\tau)} &:= 
  \max_{\substack{\sigma\in V_{j+1}(\tau_j)\\ \lvert \sigma \rvert = \ell}} \zeta_{\langle \sigma_1,\tau_1,\dots,\tau_j, \sigma\rangle}(\langle\rangle) ,
  \shortintertext{where $\ell$ is the maximal $E_{j+1}$-expansionary level in $V_{j+1}(\tau_j)$,}
  \zeta_{\langle \sigma_1,\tau_1\dots,\sigma_j\rangle}(\tau) &:=
  \begin{cases}
    0 & 
    \begin{varwidth}{5cm}
      if $\tau$ is a leaf in $E_j(\sigma_j)$\\ \hspace*{2em} and $\tau({\lvert \tau \rvert-1}) = \bot$,
    \end{varwidth} \\
    \omega^{\zeta_{\langle \sigma_1,\tau_1,\dots,\sigma_j,\tau\rangle}} &
    \begin{varwidth}{5cm}
      if $\tau$ is a leaf in $E_j(\sigma_j)$\\ \hspace*{2em} and $\tau({\lvert \tau \rvert-1}) \neq \bot$,
    \end{varwidth} \\
    \sum_{\tau'\sqsupset_{E_j(\sigma_j)} \tau} \zeta_{\langle \sigma_1,\tau_1,\dots,\sigma_j\rangle}(\tau') & \text{if $\tau$ is not a leaf in $E_j(\sigma_j)$.}
  \end{cases}
  \intertext{To the full $k$-iterated monotone enumeration we assign the following ordinal.}
\zeta_{(V_i,E_i)_{1\le i \le k}} \phantom{(\tau)}&:= \zeta_{\langle\rangle}.
\end{align*}
Note that $\zeta_{(V_i,E_i)_{1\le i \le k}} \le \omega^b_k < \omega^\omega_k$.

\begin{lemma}\label{lem:bmefor}
  Let $(V_i,E_i)_{1\le i \le k}$ be a $k$-iterated monotone enumeration and a tree $V'_1$ be given, 
  such that $V'_1$ properly extends $V_1$.
  If $V'_1$ contains strictly more $E_1$\nobreakdash-\hspace{0pt}expansionary levels thant $V_1$, then
  \[
  \zeta_{(V_i,E_i)_{1\le i \le k}}>\zeta_{{(V'_i,E_i)}_{1\le i \le k}}
  ,\]
  where for $i\ge 2$ we set $V'_i:=V_i$.
\end{lemma}
\begin{proof}
  We prove my induction that
  \begin{enumerate}[label=(\alph*)]
  \item\label{enum:l:1}
    $\zeta_{\langle \sigma_1,\tau_1,\dots,\sigma_j,\tau_j \rangle} > \zeta_{\langle \sigma_1,\tau_1,\dots,\sigma_j,\tau'_j \rangle}$,
    if $\tau'_j\sqsupseteq \tau_j$ enumerates a new $E_{j+1}$-\hspace{0pt}expansionary level in $V_{j+1}$,
  \item\label{enum:l:2}
    $\zeta_{\langle \sigma_1,\tau_1\dots,\sigma_j\rangle} > \zeta_{\langle \sigma_1,\tau_1\dots,\sigma'_j\rangle}$, if $\sigma'_j\sqsupseteq \sigma_j$ enumerates a new element into $E_{j}$.
  \end{enumerate}  
  This directly implies then the lemma.

  To prove the induction we start with \ref{enum:l:2} for $j=k$. This case follows as in \prettyref{pro:bme}.\\
  For \ref{enum:l:1} and $j$ we assume that \ref{enum:l:2} already holds for $j$. By the induction hypothesis each of the terms in the maximum in the definition $\zeta_{\langle \sigma_1,\tau_1,\dots,\sigma_j,\tau_j \rangle}$ decreases. Therefore,  $\zeta_{\langle \sigma_1,\tau_1,\dots,\sigma_j,\tau'_j \rangle} < \zeta_{\langle \sigma_1,\tau_1,\dots,\sigma_j,\tau_j \rangle}$. \\
  For \ref{enum:l:2} and $j<k$ we assume that \ref{enum:l:1} already holds for $j+1$. This case follows by a similar proof as in \prettyref{pro:bme} together with the induction hypothesis.
\end{proof}

For the backward direction we will only consider simplified iterated monotone enumerations where the trees $V_k(\tau)$ are trivial, i.e., they contain only branches of the form $\langle 0,\dots, 0, 1\rangle$, where the length codes $\tau$. Thus, we can omit the $V_i$ and assume that $E_{j+1}$ is of the form $E_{j+1}(\tau_j)$ with $\tau_j\in E_j$. 
With this the bound on the $E_1$-expansionary levels in $V_1$ then becomes a bound cardinality of $E_1$.

Further we make the assumption that each tree contains $\langle \bot \rangle$ and that $E_j(\langle \bot \rangle) = \{\bot\}$.
For ease of notation we will omit the $V_j$.

\begin{lemma}\label{lem:bmeback1}
  For any $\alpha <\omega_{k+1}^{b}$, one can effectively find an $k+1$-iterated bounded enumeration $\langle E_{1},\dots, E_{k+1} \rangle$
  where $E_1$ is bounded by $b$ and such that $\zeta_{\langle E_{1},\dots, E_{k+1} \rangle}=\alpha$.
\end{lemma}
\begin{proof}
  We will prove this lemma by induction on $k$.

  For the case $k=0$ and $\alpha=0$, set $E_{1}:=\{\langle \bot \rangle\}$.
  For $k=0$ and $\alpha >0$, write $\alpha=\sum_{1\le j\le l} \omega^{e_{j}}$ such that $b>e_{1}\ge e_{2}\ge\dots \ge e_l$.
  In this case one easily checks that, the enumeration the constant sequences $\langle j, \dots, j \rangle$ of length $b-e_j$ in $b-e_j$ steps for $j\in [1;l]$, i.e.,
  \begin{align*}
    E_{1}&:=\{\langle j \rangle^{\ast m} \mid 1\le j\le l \AND m \le b - e_j \},
    \shortintertext{such that}
    {\lvert \tau \rvert}_{E_1} &= {\lvert \tau \rvert} \quad\text{for any $\tau\in E_1$}
  \end{align*}
  is the desired tree. (We write $\langle j \rangle^{\ast m}$ for the $m$-fold repetition of $j$.)

  For the case $k>0$ and $\alpha=0$, we again set $E_{1}:=\{\langle \bot \rangle\}$, and $E_{i}(\langle \bot \rangle)=\{\langle \bot \rangle\}$ for any $j\in [1;k+1]$.
  If $\alpha>0$, write $\alpha=\sum_{1\le j\le l} \omega^{\alpha_{j}}$ such that $\omega_{k}^{b}>\alpha_{0}\ge \alpha_{1}\ge\dots\ge \alpha_l$.
  By induction hypothesis, one can find effectively $k$-iterated bounded enumerations $(E^{j}_i)_{1\le i \le k}$ such that $\zeta_{(E^{j}_i)_{1\le i \le k}}=\alpha_{j}$.

  Let  
  \begin{align*}
    E_{1} &:=\{\langle j \rangle\mid 1\le j\le l\}, \\
    E_{i+1}(\langle j \rangle \ast \tau) &:= \langle j \rangle \ast E^{j}_{i}(\tau) = \{ \langle j \rangle \ast \sigma \mid \sigma \in E^{j}_{i}(\tau) \},
  \end{align*}
  for $i\in [1;k]$. We can easily check that $\zeta_{(E_i)_{1 \le i \le l}}= \alpha$.
\end{proof}

We say that a bounded enumeration $E$ is \emph{separating} if
\[
E(\tau_1) \cap E(\tau_2) = E(\tau) \quad\text{where $\tau$ longest common initial substring of $\tau_1,\tau_2$}.
\]
In other words, $E$ is separating if different paths enumerate separate sets of strings, or each $\sigma$ is enumerated at most once into $E$.
We say that a $k$-iterated bounded enumeration $(E_i)_{1\le i \le k}$ is separating if each $E_i$ is separating.

We can make any enumeration separating by just coding into each string where it has been enumerated without changing the ordinal.

\begin{lemma}\label{lem:bmeback2}
  For any separating $k+1$-iterated bounded enumeration $(E_{i})_{1\le i \le k+1}$ bounded by $b+1$ with $\zeta_{(E_{i})_{1\le i \le k+1}}=:\alpha<\omega_{k+1}^{b}$ and for any $\beta<\alpha$, one can effectively find a separating proper monotone extension 
  $(E'_{i})_{1\le i \le k+1}$ also bounded by $b$, such that $\zeta_{(E'_{i})_{1\le i \le k+1}}\ge \beta$. 
  
  Proper extension means hear that only leafs of $E_i$ are extended in $E_i'$ and $E_1\subsetneq E_1'$.
\end{lemma}
\begin{proof}
  We will prove this lemma by induction on $k$.

  For the case $k=0$, write $\alpha=\sum_{1\le i\le l} \omega^{e_{j}}$ and $\beta=\sum_{1\le j\le l'} \omega^{f_{j}}$  such that $b>e_{0}\ge e_{1}\ge\dots \ge e_{l}$ and $b>f_{0}\ge\dots\ge f_{l'}$.\\
  If $l'<l$ and $e_{j}=f_{j}$ for all $j\le l'$, find a leaf $\tau\in E_{1}$ such that ${\lvert \tau \rvert}_{E_1}=b-e_{l'+1}$, and put $E_{1}'=E_{1}\cup\{\tau \ast \langle \bot \rangle\}$.\\
  Otherwise, there exists $j^{*}<l, l'$ such that $e_{j^{*}}>f_{j^{*}}$.
  Find a leaf $\tau\in E_{1}$ such that ${\lvert\tau\rvert}_{E_1}=b-e_{j^{*}}$.
  Let $E_{1}':=E_{1}\cup\{\tau \ast \langle j \rangle^{\ast m}\mid j^{*}\le j\le l' \AND m \le {e_{j^{*}}-f_{j*}}\}$ where $\langle j \rangle^{\ast m}$ is enumerated step by step, i.e., ${\lvert \sigma \ast \langle j \rangle^{\ast ({e_{j^{*}}-f_{j*}})}\rvert}_{E_1} = b-f_{j^{*}}$.

  For the case $k>0$,  write $\alpha=\sum_{1\le j\le l} \omega^{\alpha_{j}}$ and $\beta=\sum_{1\le j\le l'} \omega^{\beta_{j}}$  such that $\omega_{k}^{b}>\alpha_{1}\ge \dots \ge \alpha_l$ and $\omega_{k}^{b}>\beta_{1}\ge\dots \ge \beta_{l'}$.
  If $l'<l$ and $\alpha_{j}=\beta_{j}$ for all $j\le l'$.
  Find a leaf $\tau_1\in E_{1}$ such that $\zeta_{\langle  \rangle}(\tau_1)=\omega^{\alpha_{l'+1}}$.
  Set $E_{1}':=E_{1}\cup\{\tau \ast \langle \bot \rangle\}$, and set
  \begin{align*}
    \tau_{i+1} &:= \min\{\tau \mid \text{$\tau$ is leaf in $E_{i+1}(\tau_i)$ and $\zeta_{\langle \tau_1,\dots,\tau_i\rangle}(\tau)>0$}\}, \\
    E'_{i+1}(\tau) 
    &:= 
      \begin{cases}
        E_{i+1}(\tau_i) \cup \{\tau_{i+1} \ast \langle\bot\rangle\} & \text{if }\tau = \tau_i \ast \langle \bot \rangle ,   \\
        E_{i+1}(\tau) & \text{otherwise.}
      \end{cases}
  \end{align*}
  Otherwise, there exists $j^{*}<l, l'$ such that $\alpha_{j^{*}}>\beta_{j^{*}}$.
  Find a leaf $\tau_1\in E_{1}$ such that $\zeta_{\langle  \rangle}(\tau_1)=\omega^{\alpha_{j^{*}}}$.
  By induction hypothesis, there exist proper a extensions $(E'_{i})_{1\le i \le k}$ of $(E_{2}(\tau_1), E_{3},\dots, E_{k})$ such that $\zeta_{(E^*_{i})_{1\le i \le k}} \ge \beta_{j}$ for $j\in [j^*;l']$.
  (One can effectively find these extensions.) We may further assume that the new elements enumerated into $E'^{j}_i$ for different $j$ are different.
  
  Set 
  \begin{align*}
    E_{1}'& :=E_{1}\cup\{\tau_1 \ast \langle j \rangle^{\ast m}\mid j^{*}\le j\le l'\}, 
            \shortintertext{$m$ is minimal with $\zeta_{(E'^{j}_i[\lvert \tau_1\rvert + m])_{1\le i \le k}} < \alpha_j$,} 
    E_{2}'(\tau) &:=
                   \begin{cases}
                     E'^{j}_{i}[\lvert\tau \rvert] & \text{if $\tau = \tau_1 \ast \langle j \rangle$,} \\
                     E_2(\tau) & \text{otherwise,}
                   \end{cases} \\
    E'_{i+2}(\tau)&:=
    \begin{cases}
      E'^j_{i+1}(\tau) &\text{for $\tau$ being enumerated below a $\tau_1 \ast \langle j\rangle$.} \\
      E_j(\tau) & \text{otherwise.}
    \end{cases}
  \end{align*}
  The last case distinction is possible by separability. We can easily check that $(E_{i}')_{1\le i \le k+1}$ is again separable and $\alpha > \zeta_{(E_{i}')_{1\le i \le k+1}}\ge \beta$.
\end{proof}

\begin{proof}[Proof of \prettyref{thm:bme}]
  The forward direction follows directly from \prettyref{lem:bmefor} and the fact that $\zeta_{(V_i,E_i)_{1\le i\le k}} \le \omega^\omega_k$ for any $k$-iterated monotone enumeration.
  For the backward direction assume that there exists an infinite descending sequence of ordinals $(\alpha_n)_n$ with $\alpha_0=\omega^\omega_k$. Take $b$ large enough that $\alpha_1 \le \omega^b_k$.
  By \prettyref{lem:bmeback1} and the comments below it, there exists a separating $k$-iterated $b+1$-bounded monotone enumeration $(E_i^1)_{1 \le i \le k}$ with 
  $\zeta_{(E_i^1)_{1 \le i \le k}} = \alpha_1$. \prettyref{lem:bmeback2} gives a sequence $\big((E_i^n)_{1 \le i \le k}\big)_n$ of separating $k$-iterated $b$-bounded monotone enumerations with $\zeta_{(E_i^n)_{1 \le i \le k}} \ge \alpha_n$. Now set $E'_i := \bigcup_n E_i^{n}$, then $(E_i')_{1\le i \le k}$ is again $k$-iterated $b+1$-bounded monotone enumeration. However by construction $E_0$ is infinite and thus we get $\NOT\lp{BME_{\mathnormal{k}}}$.
\end{proof}

We close this section with showing that weak K\"onig's lemma, a formulation of the Baire Category theorem, and the cohesive principle are $\Pi^1_1$-conservative over $\lp{RCA_0}+ \lp{WF(\mathcal{O})}$ for each primitive recursive linear order $\mathcal{O}$. Here $\lp{WF(\mathcal{O})}$ stands for the statement that $\mathcal{O}$ is well-founded. (In particular one can take for $\mathcal{O}$ any ordinal $\alpha<\epsilon_0$.)
This shows that \lp{BME} is stable with those axioms.

\begin{theorem}[Folklore]\label{thm:wklcons}
  For each primitive recursive linear order $\mathcal{O}$, the system $\lp{WKL_0}+ \lp{WF(\mathcal{O})}$ is $\Pi^1_1$-conservative over $\lp{RCA_0} + \lp{WF(\mathcal{O})}$.
\end{theorem}
\begin{proof}
  The proof proceeds as the classical proof of the $\Pi^1_1$-conservativity of \lp{WKL_0} over \lp{RCA_0}, see \cite[IX.2]{sS09}. 
  By a standard argument it is sufficient to show that each countable model of $\lp{RCA_0} + \lp{WF(\mathcal{O})}$ can be extended to an $\omega$-submodel of $\lp{WKL_0}+\lp{WF(\mathcal{O})}$.
  This follows, again by a standard argument, from the fact that for each model $M=(\lvert M \rvert, \mathcal{S}_M)$ of $\lp{RCA_0} + \lp{WF(\mathcal{O})}$ and each tree infinite tree $T\in \mathcal{S}_M$ one can find an $\omega$-submodel $M'\models  \lp{RCA_0} + \lp{WF(\mathcal{O})}$ containing an infinite branch of $T$.
    To establish this, let $M=(\lvert M \rvert, \mathcal{S}_M)$ be a model of $\lp{RCA_0} + \lp{WF(\mathcal{O})}$.
    The model will be extended by forcing along the set $\mathcal{T}_M$ of infinite 0/1-trees in $M$ ordered by inclusion, i.e.,
    \[
    \mathcal{T}_M := \left\{\, T\in \mathcal{S}_M \sizeMid M \models \text{$T$ is an infinite subtree of $2^\Nat$} \,\right\}
    .\]
    For $T_1,T_2\in \mathcal{T}_M$ we set $T_1\ge T_2$ if{f} $T_1\supseteq T_2$. A set $\mathcal{D}\subseteq \mathcal{S}_M$ is called dense if for every $T\in \mathcal{T}_M$ there is an $T'\in \mathcal{D}$ with  $T\ge T'$. 
    A set $G$ is called $\mathcal{T}_M$-generic if{f} it meets every definable, dense subset of $\mathcal{T}_M$.

    One can show that any infinite tree in $M$ has a $\mathcal{T}_M$-generic path and that for each $\mathcal{T}_M$-generic $G$ we have that $M[G]\models I\Sigma^0_1$,
    where $M[G]:=\big(\vert M \rvert, \{X \subseteq \lvert M \rvert \mid \text{$X$ is recursive in $G$ and sets from $\mathcal{S}_M$}\}\big)$.
    See Lemmas~X.2.3--5 of \cite{sS09}.

    To prove this theorem it is thus sufficient to show the following lemma.
    \begin{lemma}\label{lem:wklcons}
      For each $M\models \lp{RCA_0} + \lp{WF(\mathcal{O})}$ and each $\mathcal{T}_M$-generic $G$, we have that $M[G]\models \lp{WF(\mathcal{O})}$.
    \end{lemma}
    \begin{proof}[Proof of \prettyref{lem:wklcons}]
      To show this lemma it is sufficient to show that the $e$-th Turing functional $\Phi^G_e$ relative to $G$ for any ($e\in \lvert M \rvert$) does not give an infinite descending chain in $\mathcal{O}$.
      
      For a $\sigma\in \lvert M \rvert$ viewed as a finite binary sequence in $M$, and $T\in \mathcal{T}_M$ we will write $\sigma\prec T$ if{f} $M\models \text{``any $\tau\in T$ is compatible with $\sigma$''}$.
      For $e,m\in\lvert M \rvert$, put
      \begin{align*}
        \mathcal{D}^1_e &:= \left\{\, T \in \mathcal{T}_M \sizeMid 
                          \Exists{n}\Exists{\sigma} \left(
                          \begin{multlined}
                          \sigma \prec T \AND
                          \Forall{i\le n} (\Phi^\sigma_{e,\lvert \sigma \rvert}[i]{\downarrow}) \\[0.5ex] \AND 
                          \text{$\big(\Phi^\sigma_{e,\lvert \sigma \rvert}[i]\big)_{i=0}^n$ is not strictly decreasing in $\mathcal{O}$}
                        \end{multlined}\right) \,\right\},\\
        \mathcal{D}^2_{e,m} & := \left\{\, T \in \mathcal{T}_M \sizeMid \Forall{\tau\in T} \left(\Phi^\tau_{e,\lvert \tau \rvert}[m]{\uparrow}\right)\,\right\}, \\
        \mathcal{D}_e &:= \mathcal{D}^1_e \cup \  \bigcup_{\mathclap{{m\in\lvert M \rvert}}} \mathcal{D}^2_{e,m}.
      \end{align*}
      Clearly, if $T\in\mathcal{D}_e$ and $G\in [T]$, then, $\Phi^G_e$ is not an infinite descending sequence of $\mathcal{O}$.
      
      Now, we want to show that $\mathcal{D}_e$ is dense. 
      Assume not then there exists an infinite tree $T\in \mathcal{T}_M$ such that any infinite subtree is not in $\mathcal{D}_e$.
      Put $l_0:=0$ and $l_{m+1}:= \min \left\{ l > l_m \sizeMid \Forall{\tau\in T\cap 2^l} (\Phi^\tau_{e,\lvert \tau \rvert}[m+1]{\downarrow})\right\}$.
      Such an $l$ always exists since there are only finitely many $\tau\in T$ such that $\Phi^\tau_{e,\lvert \tau \rvert}[m]{\uparrow}$. Otherwise they would form an infinite subtree of $T$ belonging to $\mathcal{D}^2_{e,m}$.
      Note that the sequence $l_m$ is computable in $M$.

      For each $m\in \lvert M \rvert$ and each $\tau \in T\cap 2^{l_m}$ the finite sequence $\big(\Phi^\tau_{e,\lvert \tau \rvert}[i]\big)_{i=0}^m$ is strictly decreasing in $\mathcal{O}$,
        since otherwise the subtree below $\tau$ would lie in $\mathcal{D}^1_e$. 
        Therefore, the function $f(m) := \min_{\mathcal{O}}\left\{ \phi^\tau_{e,\lvert \tau \rvert} \sizeMid \tau \in T \cap 2^{l_m} \right\}$ is computable in $M$ and one easily checks that it gives an infinite strictly decreasing sequence in $\mathcal{O}$. This contradicts the fact $M\models \lp{WF(\mathcal{O})}$ and hence $\mathcal{D}_e$ must be dense.
   \end{proof}
 \renewcommand{\qed}{}\end{proof}

The Baire Category theorem for Cantor space can be formulated in the following way.
For a $\sigma\in 2^{<\Nat}$ and $X\in 2^\Nat$ we will write $\sigma \subseteq X$ if $X$ extends $\sigma$. A set $D$ is called dense if for each $\sigma\in 2^{<\Nat}$ there is a $\tau\in D$ with $\tau\supseteq \sigma$. 
We say that $X$ meets $D$ if $\Exists{\sigma\in D} \left(\sigma \subseteq D\right)$.
The Baire category theorem (\lp{BCT}) is then the statement that  every sequence of dense sets $D_i\subseteq 2^{<\Nat}$ there exists a set $G$ that meets every $D_i$.

\begin{theorem}
  For each primitive recursive linear order $\mathcal{O}$, the system $\lp{RCA_0}+ \lp{WF(\mathcal{O})} + \lp{BCT}$ is $\Pi^1_1$-conservative over $\lp{RCA_0} + \lp{WF(\mathcal{O})}$.  
\end{theorem}
\begin{proof}
  As in the proof of \prettyref{thm:wklcons} it is sufficient to show that each countable model $M=(\lvert M \rvert, \mathcal{S}_M)$ of $\lp{RCA_0} + \lp{WF(\mathcal{O})}$ can be extended to an $\omega$-submodel of \lp{BCT}.

  In \cite[Lemma~6.2]{BS93} it is shown that one can find a $G$ such that $M[G]\models \ls{RCA_0} + \lp{BCT}$ and $G$ intersects all dense $M$-definable sets. Such a set $G$ will be called $M$-generic. The theorem follows by showing the following lemma.

  \begin{lemma}\label{lem:bctcons}
    For each $M\models \ls{RCA_0} + \lp{WF(\mathcal{O})}$ and each $M$-generic $M[G]\models \lp{WF(\mathcal{O})}$.
  \end{lemma}
  \begin{proof}[Proof of \prettyref{lem:bctcons}]
    As in \prettyref{lem:wklcons} we construct for each Turing-functional $\Phi_e^X$ a dense set $D_e$. Hence put,
    \begin{align*}
      D^1_e 
      &:= \left\{\, \sigma\in 2^{<\lvert M \rvert} \sizeMid \Exists{n\in \lvert M\rvert} \left(
      \begin{multlined}
        \Forall{i\le n} (\Phi^\sigma_{e,\lvert \sigma \rvert}[i]{\downarrow}) \AND
        \big(\Phi^\sigma_{e,\lvert \sigma\rvert}[i]\big)^n_{i=0} \\[0.5ex]
        \text{is not strictly decreasing in $\mathcal{O}$}
      \end{multlined}
      \right) \,\right\}, \\
      D^2_{e,m} &:= \left\{\, \sigma\in 2^{<\lvert M\rvert} \sizeMid \Forall{\tau \supseteq \sigma}\left(\Phi^\tau_{e,\lvert \tau \rvert}[m]{\uparrow}\right) \, \right\}, \\
      D & := D^1_e \cup\ \bigcup_{\mathclap{m\in M}}D^2_{e,m}.
    \end{align*}
    Clearly, if a generic $G$ meets $D_e$ then $\Phi^G_e$ is not an infinite descending sequence of $\mathcal{O}$.
    Now, we want to show that $D_e$ is dense.
    Assume not, then there exists a $\sigma_0$ such that for any $\sigma \supseteq \sigma_0$ we have $\Forall{i \le n}\left(\Phi^\sigma_{e,\lvert \sigma \rvert}[i]{\downarrow}\right)$ implies that $\big(\Phi^\sigma_{e,\lvert \sigma\rvert}[i]\big)_{i=0}^n$ is strictly decreasing in $\mathcal{O}$, 
    and for any $m\in \langle M \rangle$ the set $\left\{\tau \sizeMid \Phi^\tau_{e,\lvert \tau \rvert}[m]{\downarrow} \right\}$ is dense below $\sigma_0$. By the latter condition one can easy construct a computable in $M$ set $X\supseteq \sigma_0$ such that $\Phi^X_e(m){\downarrow}$ for any $m\in \lvert M \rvert$. By the former $\Phi^X_e$ outputs a strictly decreasing sequence of $\mathcal{O}$ in $M$ which is a contradiction. 
  \end{proof}
  \renewcommand{\qed}{}
\end{proof}

A sentence of the form
\[
\Forall{X} \left(\phi(X) \IMPL \Exists{Y} \eta(X,Y) \right)
\]
where $\phi$ is arithmetical and $\eta\in \Sigma^0_3$, is called \emph{restricted $\Pi^1_2$-sentence} (r-$\Pi^1_2$). Hirschfeld and Shore showed that the cohesive principle (\lp{COH}) is r-$\Pi^1_2$-conservative over $\ls{RCA_0}$, see \cite[Theorem~7.18]{dH15} and \cite{HS07}.
Assume that $\ls{RCA_0} + \lp{WF(\mathcal{O})}$ does prove a r\nobreakdash-$\Pi^1_2$\nobreakdash-\hspace{0pt}sentence, i.e.,
\begin{align*}
  \ls{RCA_0} + \lp{COH}& + \lp{WF(\mathcal{O})} 
   \vdash \Forall{X} \left(\phi(X) \IMPL \Exists{Y} \eta(X,Y) \right) .
  \intertext{By the deduction theorem this is equivalent to}
  \ls{RCA_0} + \lp{COH} 
  & \vdash \Forall{Z}\lp{WF(\mathcal{O})}[Z] \IMPL \Forall{X} \left(\phi(X) \IMPL \Exists{Y} \eta(X,Y) \right),
    \intertext{where $\lp{WF(\mathcal{O})}[Z]$ denotes each $Z$-computable sequence is well-founded. Note that this can be written as a $\Sigma^0_2[Z]$-formula.
    By logical transformation this is equivalent to}
    \ls{RCA_0} + \lp{COH} &\vdash \Forall{X} \left(\phi(X) \IMPL \Exists{Y} \Exists{Z} \left(\eta(X,Y) \OR \NOT \lp{WF}[Z] \right)\right).
\end{align*}
Since this is again a r\nobreakdash-$\Pi^1_2$\nobreakdash-sentence we can apply the above-mentioned result.
This proves the following theorem.
\begin{theorem}
  For each primitive recursive linear order $\mathcal{O}$, the system $\lp{RCA_0}+ \lp{WF(\mathcal{O})} + \lp{COH}$ is r-$\Pi^1_2$-conservative over $\lp{RCA_0} + \lp{WF(\mathcal{O})}$.  
\end{theorem}

\section{Conclusion}

We have shown in \prettyref{thm:m} that \lp{WF(\omega^\omega)} has many equivalent formulations and occurs far more often than expected. It has been rediscovered in different contexts, see for instance \cite{HP86} and \cite{CSY14} as already mentioned above. This shows that there are only a few natural first-order principles between $I\Sigma_1$ and $I\Sigma_2$, and \lp{WF(\omega^\omega)} has to be considered one of them, besides induction and bounded collection principles.
For this reason we believe that the usually in reverse mathematics considered Kirby-Paris hierarchy as shown in \prettyref{fig:pariskirby} has to be extended to give a comprehensive picture. \prettyref{fig:exth} displays such an extension by \lp{WF(\omega^\omega)}.
\begin{figure}[tb]
  \centerline{
  \xymatrix@1{
    I\Sigma_1 \ar@{=>}[r] & B\Sigma_2 \ar@{=>}[r] & I\Sigma_2 \ar@{=>}[r] & B\Sigma_3 \ar@{=>}[r] & \cdots
  }
  }
  \caption{Pairs-Kirby hierarchy}\label{fig:pariskirby}
\end{figure}
\begin{figure}[tb]
  \centerline{
  \xymatrix@1{
    & B\Sigma_2 \ar@{=>}[dr] \ar|-*=0@{||}@{<->}[dd]\\
    I\Sigma_1  \ar@{=>}[dr] \ar@{=>}[ur] & & B\Sigma_2 + \lp{WF(\omega^\omega)} \ar@{=>}[r] & I\Sigma_2  \ar@{=>}[r] & B\Sigma_3 \ar@{=>}[r] & \cdots  \\
    & \lp{WF(\omega^\omega)} \ar@{=>}[ur]
  }}
  \caption{Extended Paris-Kirby hierarchy}\label{fig:exth}
\end{figure}
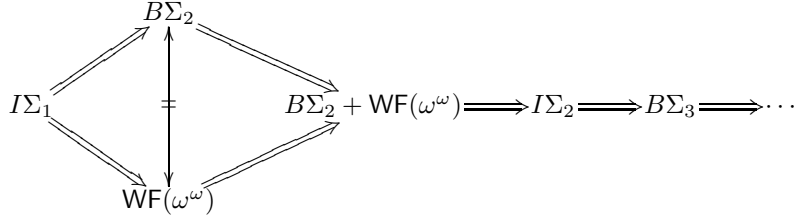
This hierarchy has been defined in \cite{HP98}. There the considered formulation of \lp{WF(\omega^\omega)} was $P\Sigma_1$, and more generally $P\Sigma_{n+1}$ for all $n$ was considered. As mentioned above $P\Sigma_{n+1}$ lies between $I\Sigma_{n+1}$ and $I\Sigma_{n+2}$. However, a similar equivalence as in \prettyref{thm:m} for $P\Sigma_{n+1}$ with $n>0$ cannot hold for quantifier reasons. In detail, $P\Sigma_2$ is $\Pi_4$ while $\lp{WF(\omega^\omega_2)}$ is still $\Pi_3$. Thus, it is unlikely to find similar extensions between $I\Sigma_{n+1}$ and $I\Sigma_{n+2}$ that are equally natural.

\smallskip
We furthermore characterize the principles \lp{BME} and \lp{BME_{\mathnormal{n}}} in terms of well-foundedness of ordinals.
This allows us to answer the question whether Ramsey's theorem for pairs and two colors (\lp{RT^2_2}) implies \lp{BME}, as ask by Chong, Slaman, and Yang in \cite[Question~5.2]{CSYtaInd}, negatively. This cannot be the case since it is known that \lp{RT^2_2} is $\Pi^1_1$\nobreakdash-conservative over $I\Sigma_2$, see \cite{CJS01}, where $\omega^\omega_3$ cannot be seen to be well-founded in $I\Sigma_2$. Thus $\lp{RCA_0}+\lp{RT^2_2} \nvdash \lp{BME_3}$.

Let $\lp{ME}$ be the \emph{monotone enumeration principle} which states that each unbounded monotone enumeration has an infinite branch. This principle is formalized in \ls{RCA_0}. For \lp{ME} we have the following well-known result.
\begin{theorem}[Folklore, \ls{RCA_0}] 
  \ls{ACA_0} and \lp{ME} are equivalent.
\end{theorem}
\lp{BME} can be seen as a miniaturization of \lp{ME} as  certain iterations of the Paris-Harrington principle are for Ramsey's theorem for pairs, see \cite{BW,kY13,kYta}, or has been done for $P\Sigma_1$ in \cite{HP86}. For the Paris-Harrington principle equivalences between these miniaturizations and the provably recursive functions (in some cases even provable $\Pi^0_3$ or $\Pi^0_4$ statements) of \lp{RT^2_2} over different systems (in detail \ls{WKL_0}, \ls{WKL_0^*}) have been established. For $P\Sigma_1$ this has not been done yet.
Our characterization in \prettyref{thm:m} together with Theorem~3 of \cite{HP86} shows that the miniaturization of \cite{HP86} is faithful, in the sense that they prove the same $\Pi^0_2$\nobreakdash-sentences.
Since \prettyref{thm:bme} shows that the $\Pi^0_3$\nobreakdash-sentences of \lp{BME} are exactly the same as of \lp{PA}, \lp{BME} is a faithful miniaturization of \lp{ME} in the same way.

\bibliographystyle{amsplain}
\bibliography{bib}

\providecommand{\bysame}{\leavevmode\hbox to3em{\hrulefill}\thinspace}
\providecommand{\MR}{\relax\ifhmode\unskip\space\fi MR }
\providecommand{\MRhref}[2]{%
  \href{http://www.ams.org/mathscinet-getitem?mr=#1}{#2}
}
\providecommand{\href}[2]{#2}
\begin{thebibliography}{10}

\bibitem{BW}
Andrey Bovykin and Andreas Weiermann, \emph{The strength of infinitary
  {R}amseyan principles can be accessed by their densities}, accepted for
  publication in Ann. Pure Appl. Logic,
  \url{http://logic.pdmi.ras.ru/~andrey/research.html}, 2005.

\bibitem{BS93}
Douglas~K. Brown and Stephen~G. Simpson, \emph{The {B}aire category theorem in
  weak subsystems of second-order arithmetic}, J. Symbolic Logic \textbf{58}
  (1993), no.~2, 557--578. \MR{1233924}

\bibitem{CJS01}
Peter~A. Cholak, Carl~G. Jockusch, Jr., and Theodore~A. Slaman, \emph{On the
  strength of {R}amsey's theorem for pairs}, J. Symbolic Logic \textbf{66}
  (2001), no.~1, 1--55. \MR{1825173}

\bibitem{CSYtaInd}
C.~T. Chong, Theodore~A. Slaman, and Yue Yang, \emph{The inductive strength of
  {R}amsey's theorem for pairs}, 2014, Preprint.

\bibitem{CSY14}
\bysame, \emph{The metamathematics of {S}table {R}amsey's {T}heorem for
  {P}airs}, J. Amer. Math. Soc. \textbf{27} (2014), no.~3, 863--892.
  \MR{3194495}

\bibitem{CLY10}
Chitat~T. Chong, Steffen Lempp, and Yue Yang, \emph{On the role of the
  collection principle for {$\Sigma^0_2$}-formulas in second-order reverse
  mathematics}, Proc. Amer. Math. Soc. \textbf{138} (2010), no.~3, 1093--1100.
  \MR{2566574}

\bibitem{HP86}
Petr H{\'a}jek and Jeff Paris, \emph{Combinatorial principles concerning
  approximations of functions}, Arch. Math. Logik Grundlag. \textbf{26}
  (1986/87), no.~1-2, 13--28. \MR{881278}

\bibitem{HP98}
Petr H{\'a}jek and Pavel Pudl{\'a}k, \emph{Metamathematics of first-order
  arithmetic}, Perspectives in Mathematical Logic, Springer-Verlag, Berlin,
  1998, Second printing. \MR{1748522}

\bibitem{HSta}
Kostas Hatzikiriakou and Stephen~G. Simpson, \emph{Reverse mathematics, young
  diagrams, and the ascending chain condition}, 2015, {\tt arXiv:1510.03106}.

\bibitem{dH15}
Denis~R. Hirschfeldt, \emph{Slicing the truth}, Lecture Notes Series. Institute
  for Mathematical Sciences. National University of Singapore, vol.~28, World
  Scientific Publishing Co. Pte. Ltd., Hackensack, NJ, 2015, On the computable
  and reverse mathematics of combinatorial principles, Edited and with a
  foreword by Chitat Chong, Qi Feng, Theodore A. Slaman, W. Hugh Woodin and Yue
  Yang. \MR{3244278}

\bibitem{HS07}
Denis~R. Hirschfeldt and Richard~A. Shore, \emph{Combinatorial principles
  weaker than {R}amsey's theorem for pairs}, J. Symbolic Logic \textbf{72}
  (2007), no.~1, 171--206. \MR{2298478}

\bibitem{KS81}
Jussi Ketonen and Robert Solovay, \emph{Rapidly growing {R}amsey functions},
  Ann. of Math. (2) \textbf{113} (1981), no.~2, 267--314. \MR{607894}

\bibitem{sS88}
Stephen~G. Simpson, \emph{Ordinal numbers and the {H}ilbert basis theorem}, J.
  Symbolic Logic \textbf{53} (1988), no.~3, 961--974. \MR{961012}

\bibitem{sS09}
\bysame, \emph{Subsystems of second order arithmetic}, second ed., Perspectives
  in Logic, Cambridge University Press, Cambridge, 2009. \MR{2517689}

\bibitem{kYta}
Keita Yokoyama, \emph{Finite iterations of infinite and finite {R}amsey’s
  theorem}, in preparation.

\bibitem{kY13}
\bysame, \emph{On the strength of {R}amsey's theorem without
  {$\Sigma_1$}-induction}, Math. Log. Q. \textbf{59} (2013), no.~1-2, 108--111.
  \MR{3032429}

\end{thebibliography}

\end{document}